\documentclass{amsart}
\usepackage{amsmath,amssymb}
\usepackage{graphicx}
\newtheorem{theorem}{Theorem}[section]
\newtheorem{proposition}[theorem]{Proposition}

\theoremstyle{definition}

\theoremstyle{remark}

\numberwithin{equation}{section}

\newcommand{\be}{\beta}
\newcommand{\de}{\delta}
\newcommand{\ep}{\epsilon}

\newcommand{\ga}{\gamma}

\newcommand{\la}{\lambda}

\newcommand{\vp}{\varphi}

\newcommand{\De}{\Delta}
\newcommand{\Ga}{\Gamma}

\newcommand{\Si}{\Sigma}
\newcommand{\Om}{\Omega}
\newcommand{\Omc}{{\Omega^c}}

\def\RR{\mathbb{R}}
\def\BB{\mathbb{B}}
\def\ZZ{\mathbb{Z}}

\renewcommand\SS{\mathbb{S}}
\newcommand{\cH}{{\mathcal H}}

\newcommand{\cP}{{\mathcal P}}

\newcommand{\cW}{{\mathcal W}}

\newcommand{\pd}{\partial}
\newcommand\minus\backslash

\newcommand\lan\langle
\newcommand\ran\rangle

\newcommand{\Span}{\operatorname{span}}

\renewcommand\leq\leqslant
\renewcommand\geq\geqslant
\newlength{\intwidth}

\newcommand\BOm{\overline\Om}

\begin{document}

\title{Eigenfunctions with prescribed nodal sets}

\author{Alberto Enciso}
\address{Instituto de Ciencias Matem\'aticas, Consejo Superior de
  Investigaciones Cient\'\i ficas, 28049 Madrid, Spain}
\email{aenciso@icmat.es}

\author{Daniel Peralta-Salas}
\address{Instituto de Ciencias Matem\'aticas, Consejo Superior de
  Investigaciones Cient\'\i ficas, 28049 Madrid, Spain}
\email{dperalta@icmat.es}

%%    General info
%\subjclass[2010]{35B38, 58J05, 58K45}
%\date{\today}
%
%\keywords{ }
%
\begin{abstract}
  In this paper we consider the problem of prescribing the nodal set
  of low-energy eigenfunctions of the Laplacian. Our main result is that, given
  any separating closed hypersurface $\Si$ in a compact $n$-manifold
  $M$, there is a Riemannian metric on $M$ such that the nodal set of
  its first nontrivial eigenfunction is 
  $\Si$. We present a number of variations on this
  result, which enable us to show, in particular, that the first
  nontrivial eigenfunction can have as many non-degenerate critical
  points as one wishes.
\end{abstract}
\maketitle

\section{Introduction}
\label{S.intro}

The eigenfunctions of the
Laplacian on a closed $n$-dimensional Riemannian manifold $(M,g)$ satisfy the equation
\[
\De u_k=-\la_k u_k\,,
\]
where $0=\la_0<\la_1\leq\la_2\leq\dots$
are the eigenvalues of $M$. In this paper we will be concerned with the geometry of the nodal sets
$u_k^{-1}(0)$ of the eigenfunctions of the Laplacian, which is a
classic topic in geometric analysis with a number of important open
problems~\cite{Ya82,Ya93}. 

Since the first nontrivial eigenfunction is
$u_1$, we will be especially interested in the shape of the nodal set
$u_1^{-1}(0)$. More generally, this paper focuses on the study of the
nodal set of low-energy eigenfunctions; in particular, we will not consider the measure-theoretic properties of the
nodal set $u_k^{-1}(0)$ as $k\to\infty$, which is an important topic that has been thoroughly
studied e.g.\ in~\cite{DF,HS,Jak}.

In fact, the central question that we will address in this paper is the following: given
a hypersurface (i.e., a codimension-1 submanifold) $\Si$ of $M$, is there a Riemannian metric $g$ such
that $\Si$ is the nodal set $u_1^{-1}(0)$?
In the case of the unit two-dimensional sphere, a detailed study of the
possible configurations of the nodal sets of the eigenfunctions has
been carried out by Eremenko, Jakobson and Nadirashvili~\cite{EJN}. In any closed surface, Komendarczyk~\cite{Komen} has
shown that, given any homotopically trivial closed curve $\ga$, there is a
metric such that $u_1^{-1}(0)$ is diffeomorphic to~$\ga$. This result has been extended to arbitrary curves on surfaces
by Lisi~\cite{Lisi} using techniques from contact topology.

Our main theorem asserts that, given an $n$-manifold $M$ and any
closed (i.e., compact boundaryless)  hypersurface $\Si\subset M$, there is a metric $g$ on $M$ such
that $\Si$ is a connected component of the nodal set of the eigenfunction $u_1$. Moreover, if~$\Si$ separates (that is, if the
complement $M\minus \Si$ is the union of two disjoint open sets), then
one can show that the nodal set does not have any other connected
components. More precisely, we have the following
statement. Throughout, we will assume that the hypersurfaces are all
connected, $n\geq3$ and all objects are of class~$C^\infty$.

\begin{theorem}\label{T.main}
Let $\Si$ be a closed orientable hypersurface of $M$. Then there
exists a Riemannian metric $g$ on~$M$ such that $\Si$ is a
connected component of the nodal set $u_1^{-1}(0)$. If $\Si$
separates, then the nodal set is exactly $\Si$.
\end{theorem}

An analogous result for the first $l$ eigenfunctions will be proved in
Theorem~\ref{T.leigen} and Proposition~\ref{P.different}; however, in
this Introduction we have chosen to restrict our attention to the
first nontrivial eigenfunction to keep the statements as simple as
possible. Somewhat related results on level sets with prescribed
topologies were derived, using completely different methods, for
Green's functions in~\cite{JDG} and for harmonic functions in $\RR^n$
in~\cite{Adv}. It should also be noted that the results that we prove
in this paper are robust in the sense that if $g$ is the metric with
eigenfunctions of prescribed nodal sets that we construct in this
paper, then any other metric close enough to $g$ in the $C^{2}$ norm
possesses the same property. In particular, the metric can be taken
analytic whenever the manifold is analytic.

The strategy of the proof of Theorem~\ref{T.main} is quite versatile and can
be used to derive a number of related results. For instance, an easy
application of the underlying philosophy enables us to prove that,
given any $n$-manifold $M$, there is a metric such that the
eigenfunction $u_1$ has as many isolated critical points as one
wishes:

\begin{theorem}\label{T.cp}
Given any positive integer $N$, there is a Riemannian metric on $M$
whose eigenfunction $u_1$ has at least $N$ non-degenerate critical points. 
\end{theorem}

An analog of this result for the first $l$ nontrivial eigenfunctions is given in
Proposition~\ref{P.critical}. It is worth recalling that, on surfaces, Cheng~\cite{Cheng} gave a
topological bound for the number of critical points of the
$k^{\mathrm{th}}$ eigenfunction that lie on the nodal line.

The ideas of the proof of the main theorem remain valid in the case of
manifolds with boundary. Specifically, let now $(M,g)$ be a compact Riemannian $n$-manifold with
boundary and consider the sequence of its Dirichlet eigenfunctions,
which with some abuse of notation we still denote by $u_k$ and satisfy
the equation
\[
\De u_k=-\la_k u_k\quad \text{in } M\,,\qquad u_k|_{\pd M}=0\,.
\]
Now $0<\la_1<\la_2\leq \dots$ are the Dirichlet eigenvalues of $M$. The problem under consideration is 
related to Payne's classical conjecture, which asserts that when $M$
is a bounded simply connected planar domain the nodal line of $u_2$ is
an arc connecting two distinct points of the boundary. Payne's
conjecture is known to hold for convex domains~\cite{Melas,Alessandrini}, but the
general case is still open. Yau~\cite[Problem~45]{Ya93} raised the question of the
validity of Payne's conjecture in $n$-manifolds with boundary. In this
direction, Freitas~\cite{Freitas} showed that there is a metric on the
two-dimensional ball for which Payne's conjecture does not hold, as
the nodal set is a closed curve contained in the interior of the
ball. The techniques used in this paper readily yield a powerful
higher-dimensional analog of this result. For simplicity and in view
of Payne's conjecture, we state it in the case of the $n$-ball,
although a totally analogous statement holds true in any compact
$n$-manifold with boundary (see Proposition~\ref{P.Payne}):

\begin{theorem}\label{T.Payne}
  Let $\Si$ be a closed orientable hypersurface contained in the
  $n$-dimen\-sional ball $\BB^n$. Then   there exists a Riemannian metric~$g$ on~$\overline{\BB^n}$ such that
  the nodal set of its second Dirichlet eigenfunction is $\Si$.
\end{theorem}

\section{Proof of the main theorem}
\label{S.main}

In this section we will prove Theorem~\ref{T.leigen} below, a consequence of
which is Theorem~\ref{T.main}. The gist of the proof, which uses ideas
introduced by Colin de Verdi\`ere ~\cite{Colin} to prescribe the multiplicity of the
first nontrivial eigenvalue $\la_1$, is to choose the metric so that
the low-energy eigenvalues are simple and the corresponding
eigenfunctions are close, in a suitable sense, to functions whose
nodal set is known explicitly. 

From now on and until Subsection~\ref{SS.Payne}, $M$ will be a compact
manifold without boundary of dimension $n\geq3$.

\subsubsection*{Step 1: Definition of the metrics} Consider a small
neighborhood $\Om\subset M$ of the orientable hypersurface $\Si$, which we can
identify with $(-1,1)\times\Si$. Let us take a metric $g_0$ on $M$
whose restriction to $\Om$ is
\[
g_0|_\Om=dx^2+g_\Si\,,
\]
where $x$ is the natural coordinate in $(-1,1)$ and $g_\Si$ is a
Riemannian metric on~$\Si$. We can assume that the first nontrivial eigenvalue of
the Laplacian on $\Si$ defined by $g_\Si$ is larger than $l^2\pi^2/4$, where
$l$ is a positive integer.

It is then clear that the first $l+1$ Neumann eigenfunctions of the
domain~$\Om$ can be written as
\[
v_k=|\Si|^{-\frac12} \cos \frac{k\pi(x+1)}2\,,\qquad 0\leq k\leq l\,,
\]
where $|\Si|$ stands for the area of~$\Si$. Observe that, with this
normalization, 
\[
\int_\Om v_jv_k=\de_{jk}\,.
\]
Denoting by $\De_0$ the Laplacian corresponding to the metric $g_0$,
these eigenfunctions satisfy the equation
\[
\De_0 v_k=-\mu_kv_k\quad \text{in }\Om\,,\qquad \pd_\nu v_k|_{\pd \Om}=0
\]
with $\mu_k:=k^2\pi^2/4$.

For each $\ep>0$, let us define a piecewise smooth metric $g_\ep$ on~$M$ by
setting
\[
g_\ep:=\begin{cases}
g_0 & \text{in }\overline \Om\,,\\
\ep g_0 & \text{in }M\minus \Om\,.
\end{cases}
\]
To define the spectrum of this discontinuous metric
one resorts to the quadratic form
\begin{align*}
Q_\ep(\vp)&:=\int_M |d \vp|_\ep^2\, dV_\ep\\
&=\int_\Om |d \vp|^2+ \ep^{\frac n2-1}\int_\Omc |d \vp|^2
\end{align*}
together with the natural $L^2$ norm corresponding to the metric $g_\ep$:
\begin{align*}
\|\vp\|_\ep^2&:=\int_M \vp^2\, dV_\ep\\
&=\int_\Om \vp^2+ \ep^{\frac n2}\int_\Omc \vp^2
\end{align*}
Here the possibly disconnected set $\Omc:=M\minus\overline\Om$ stands for the interior of the complement of $\Om$,
the subscripts~$\ep$ refer to quantities computed with respect to the
metric $g_\ep$ and we are omitting the subscripts (and indeed the
measure in the integrals) when the quantities correspond to the
reference metric $g_0$. As is well known, the domain of the quadratic
form $Q_\ep$ can be taken to be the Sobolev space $H^1(M)$ (recall
that, $M$ being compact, this Sobolev space is independent of the
smooth metric one uses to define it).

By the min-max principle, the $k^{\mathrm{th}}$ eigenvalue
$\la_{k,\ep}$ of this quadratic form is
\begin{equation}\label{minmax}
\la_{k,\ep}=\inf_{W\in\cW_k} \max_{\vp\in W\minus\{0\}} q_\ep(\vp)\,,
\end{equation}
where $\cW_k$ stands for the set of $(k+1)$-dimensional linear subspaces
of $H^1(M)$ and
\begin{equation}\label{qep}
q_\ep(\vp):=\frac{Q_\ep(\vp)}{\|\vp\|_\ep^2}
\end{equation}
is the Rayleigh quotient. The $k^{\mathrm{th}}$ eigenfunction $u_{k,\ep}$ is
then a minimizer of the above variational problem for
$\la_{k,\ep}$, in the sense that any subspace that minimizes the variational
problem can be written as $\Span\{ u_{0,\ep},\dots, u_{k,\ep}\}$.

\subsubsection*{Step 2: $\la_{k,\ep}$ is almost upper bounded by the corresponding
  Neumann eigenvalue}

Let us now show that the $k^{\mathrm{th}}$ eigenvalue of the quadratic
form $Q_\ep$ is upper bounded by the Neumann eigenvalues of $\Om$ as
\begin{equation}\label{upper}
\limsup_{\ep\searrow0}\la_{k,\ep}\leq  \mu_k\,.
\end{equation}
To see this, consider the function $\psi_k\in H^1(M)$ given by
\[
\psi_k:= \begin{cases}
v_k &\text{in } \BOm\,,\\
\hat v_k &\text{in } \Omc\,,
\end{cases}
\]
where $\hat v_k$ is the harmonic extension to $\Omc$ of the Neumann
eigenfunction $v_k$, defined as the solution to
the boundary value problem
\[
\De_0\hat v_k=0\quad\text{in } \Omc\,, \qquad \hat v_k|_{\pd\Om}=v_k\,.
\]
Standard elliptic estimates show that the Sobolev norms of $\hat v_k$
are controlled in terms of those of $v_k|_{\pd\Om}$; in particular
\begin{equation}\label{boundhvk}
\int_\Omc |d\hat v_k|^2+\int_\Omc \hat v_k^2\leq C\|v_k\|_{H^{\frac12}(\pd\Om)}^2\leq C\| v_k\|_{H^1(\Om)}^2=C(\mu_k+1)\,,
\end{equation}
where $C$ is a constant independent of~$\ep$ and~$k$, and we have used
that the Neumann eigenfunctions are normalized. 

We are now ready to derive the upper bound for $\la_{k,\ep}$. We will
find it convenient to denote by $O(\ep^m)$ a quantity that is bounded
(possibly not uniformly in~$k$)  by $C\ep^m$.
From~\eqref{boundhvk} it stems that for any linear
combination $\vp,\vp'\in \Span\{\psi_1,\dots,\psi_k\}$ we have
\begin{align}
\int_M \vp\,\vp'\, dV_\ep&= \int_\Om \vp\,\vp'+ \ep^{\frac n2}\int_\Omc
\vp\, \vp'\notag\\
& =\int_\Om \vp\, \vp'+O(\ep^{\frac n2})\,
\|\vp\|_{H^1(\Om)}\|\vp'\|_{H^1(\Om)} \notag\\
& =\int_\Om \vp\, \vp'+O(\ep^{\frac n2})\, (1+\mu_k)
\|\vp\|_{L^2(\Om)}\|\vp'\|_{L^2(\Om)} \notag\\
&=\int_\Om \vp\, \vp'+O(\ep^{\frac n2})\,
\|\vp\|_{L^2(\Om)}\|\vp'\|_{L^2(\Om)}\,. \label{boundL2}
\end{align}
To pass to the last line we have simply used that, for any finite $k$,
we can absorbe the constant
$1+\mu_k$ in the $O(\ep^{\frac n2})$ term. An immediate consequence of
this inequality is that, for all $k\leq l$ and sufficiently small
$\ep$ (depending on $l$), the linear space
\[
\Span\{\psi_1,\dots,\psi_k\}
\]
is a $k$-dimensional subspace of $H^1(M)$.

A similar argument using~\eqref{boundhvk} allows us to estimate $Q_\ep(\vp)$, with $\vp\in
\Span\{\psi_1,\dots,\psi_k\}$, as
\begin{align}
Q_\ep(\vp)&=\int_\Om |d \vp|^2+ \ep^{\frac n2-1}\int_\Omc |d \vp|^2\notag\\
& =\int_\Om |d \vp|^2+O(\ep^{\frac n2-1})\, (1+\mu_k)
\|\vp\|_{L^2(\Om)}^2\notag\\
&=\int_\Om |d \vp|^2+O(\ep^{\frac n2-1})\,
\|\vp\|_{L^2(\Om)}^2\,. \label{boundH1}
\end{align}
In view of the min-max formulation~\eqref{minmax} and omitting the condition that $\vp\neq0$ for notational
simplicity, we then obtain from~\eqref{boundL2} and~\eqref{boundH1} that
\begin{align*}
\la_{k,\ep}&\leq \max_{\vp\in \Span\{\psi_1,\dots,\psi_k\}}
q_\ep(\vp)\\
&=\max_{\vp\in \Span\{\psi_1,\dots,\psi_k\}}\frac{\int_\Om |d \vp|^2+O(\ep^{\frac n2-1})\,
\|\vp\|_{L^2(\Om)}^2}{[1+O(\ep^{\frac n2})]\,
\|\vp\|_{L^2(\Om)}^2}\\
&=[1+O(\ep^{\frac n2})]\max_{\chi\in \Span\{v_1,\dots,v_k\}}
q_\Om(\chi)+O(\ep^{\frac n2-1})\\
&=[1+O(\ep^{\frac n2})]\, \mu_k +O(\ep^{\frac n2-1})\,,
\end{align*}
which proves~\eqref{upper}. Here 
\[
q_\Om(\chi):=\frac{\int_\Om|d\chi|^2}{\int_\Om \chi^2}
\]
is the Rayleigh quotient in $\Om$ and we recall that the bounds for the $O(\ep^m)$
terms are {\em not}\/  uniform in $k$.

\subsubsection*{Step 3: Convergence to the Neumann eigenfunctions}

Let us consider the linear space
\[
\cH_2:=\big\{\vp\in H^1(M): \vp|_\Omc \in H^1_0(\Omc)\,,\; \vp|_\Om=0\big\}\,.
\]
This is a closed subspace of $H^1(M)$, so it is standard that there is
another closed subspace $\cH_1$ of $H^1(M)$ such that
\begin{equation}\label{directsum}
H^1(M)=\cH_1\oplus \cH_2
\end{equation}
and
\begin{equation}\label{ortho}
\int_M g_0( \nabla \vp_1,\nabla \vp_2 ) =0
\end{equation}
for all $\vp_j\in \cH_j$. Here the gradient is computed using the
reference metric $g_0$. A short computation shows that in fact one has
\[
\cH_1:=\big\{\vp\in H^1(M): \De_0\vp=0\;\text{in }\Omc\big\}\,.
\]
We will denote by $\cP_j$ the projector associated with the subspace $\cH_j$.

Our goal now is to analyze the first $l$ eigenfunctions of
$Q_\ep$ for small $\ep$. Hence, let us fix some nonnegative
integer $k\leq l$. By the upper bound~\eqref{upper}, if $\ep$ is small
enough we have that
\[
\la_{k,\ep}< \mu_k+1\,,
\]
which implies that
\begin{equation}\label{lakep1}
\la_{k,\ep}=\inf_{W\in\cW_k'} \max_{\vp\in W\minus\{0\}} q_\ep(\vp)\,.
\end{equation}
The difference with~\eqref{minmax} is that now $\cW_k'$ is the set of
$(k+1)$-dimensional subspaces of~$H^1(M)$ such that 
\[
q_\ep(\vp)< \mu_k+1
\]
for all nonzero $\vp\in \cW_k'$.

A further simplification is the following. Let us use the direct sum
decomposition~\eqref{directsum} to write $\vp$
as
\[
\vp= \vp_1+ \vp_2\,,
\]
where we will henceforth use the notation $\vp_j:=\cP_j\vp$. The observation now is that if
\[
\|\vp_2\|_\ep^2 \geq c\|\vp\|_\ep^2
\]
for some $c>0$, then 
\begin{align*}
q_\ep(\vp)&\geq \frac{\ep^{\frac
    n2-1}\int_\Omc|d\vp_2|^2}{\|\vp\|_\ep^2}\\
&\geq \frac{c\ep^{\frac
    n2-1}\int_\Omc|d\vp_2|^2}{\|\vp_2\|_\ep^2}\\
&\geq \frac {c}\ep\frac{\int_\Omc|d\vp_2|^2}{\int_\Omc\vp_2^2}\\
& \geq \frac{c_0c}\ep\,,
\end{align*}
where the positive, $\ep$-independent constant $c_0$ is the first Dirichlet eigenvalue of $\Omc$ with the reference
metric $g_0$. Hence we easily infer that, if $\vp$ is in a subspace belonging
to $\cW_k'$ with $k\leq l$, then necessarily
\begin{equation}\label{vp2}
\|\vp_2\|_\ep^2 \leq O(\ep)\,\|\vp\|_\ep^2\,.
\end{equation}
% An immediate consequence of this is that, for small enough $\ep$,
% $\cP_1W$ is a $k$-dimensional linear space for all $W\in \cW_k'$.

By mimicking the proof of the estimate~\eqref{boundhvk}, we readily
find that if $q_\ep(\vp)\leq \mu_k+1$, one has that
\begin{equation}\label{vp1}
\int_\Omc\vp_1^2+ \int_\Omc |d\vp_1|^2\leq C \int_\Om\vp_1^2\,.
\end{equation}
We can now use the orthogonality relation~\eqref{ortho} to write, for
any nonzero $\vp\in W$ with $W\in \cW_k'$,
\begin{align*}
q_\ep(\vp)&=\frac{\int_M|d\vp_1|^2_\ep \, dV_\ep +
  \int_M|d\vp_2|^2_\ep \, dV_\ep}{\|\vp\|^2_\ep}\\
&=[1+O(\ep^{\frac12})]\frac{\int_M|d\vp_1|^2_\ep \, dV_\ep +
  \int_M|d\vp_2|^2_\ep \, dV_\ep}{\|\vp_1\|^2_\ep}\\
&\geq [1+O(\ep^{\frac12})]\frac{\int_M|d\vp_1|^2_\ep \, dV_\ep
}{\|\vp_1\|^2_\ep}\\
&= [1+O(\ep^{\frac12})]\frac{\int_\Om|d\vp_1|^2 + \ep^{\frac n2-1}\int_\Omc|d\vp_1|^2
}{\int_\Om\vp_1^2+ \ep^{\frac n2}\int_\Omc\vp_1^2}\\
&= [1+O(\ep^{\frac12})]\frac{\int_\Om|d\vp_1|^2 + O(\ep^{\frac n2-1})\int_\Om\vp_1^2
}{\int_\Om\vp_1^2}\\
&= [1+O(\ep^{\frac12})]\, q_\Om(\vp_1|_\Om)+O(\ep^{\frac n2-1})\,.
\end{align*}
Here we have used the inequalities~\eqref{vp2}-\eqref{vp1}, which in
particular imply 
\[
\|\vp\|_\ep=[1+O(\ep^{\frac12})]\|\vp_1\|_\ep\,.
\]

Therefore, by the min-max principle~\eqref{lakep1},
\begin{align*}
\la_{k,\ep}&\geq [1+O(\ep^{\frac12})]\inf_{W\in\cW_k'} \max_{\vp\in
  W\minus\{0\}} q_\Om(\vp_1|_\Om) +O(\ep^{\frac n2-1})\\
&\geq [1+O(\ep^{\frac12})]\, \mu_k+O(\ep^{\frac n2-1})\,.
\end{align*}
Together with the upper bound~\eqref{upper}, this shows that
\[
\lim_{\ep\searrow0}\la_{k,\ep}=\mu_k\,.
\]
Since the eigenvalues $\mu_k$ are simple for $k\leq l$, it is
well-known that this implies that there are nonzero constants $\be_k$
such that the restriction to $\Om$ of the
eigenfunctions $u_{k,\ep}$ converges in $H^1(\Om)$ to the Neumann
eigenfunctions $\be_kv_k$ for all
$k\leq l$, that is,
\begin{equation}\label{nueva}
\lim_{\ep\searrow0} \|u_{k,\ep}- \be_k v_k\|_{H^1(\Om)}=0\quad
\text{for }
k\leq l\,. 
\end{equation}
For simplicity of notation, we will redefine the normalization
constants of the eigenfunctions $u_{k,\ep}$ if necessary to take $\be_k=1$.
In view of~\eqref{nueva}, standard elliptic estimates then ensure that $u_{k,\ep}\to v_k$ as
$\ep\to0$ in $C^m(K)$, for any compact subset $K$ of $\Om$ and any integer
$m$.

\subsubsection*{Step 4: Characterization of the nodal sets}

Let us fix a small enough $\ep>0$ and take a sequence of uniformly
bounded smooth functions $\chi_{j,\ep}\in C^\infty(M)$ converging pointwise to:
\[
\lim_{j\to\infty} \chi_{j,\ep}(x)=\begin{cases}
1 &\text{if } x\in\BOm\,,\\
\ep &\text{if } x\in\Omc\,.
\end{cases}
\]
It is then well-known (see e.g.~\cite{BU83}) that
the $k^{\mathrm{th}}$ eigenvalue of the Laplacian corresponding to the
metric $g_{j,\ep}:=\chi_{j,\ep}g_0$ then converges to $\la_{k,\ep}$ as
$j\to\infty$, and that its $k^{\mathrm{th}}$ eigenfunction
\[
u_{k,\ep}^j
\]
converges
in $H^1(M)$ (and in $C^m(K)$ for any compact subset of $M\backslash\partial\Omega$ and any
$m$) to $u_{k,\ep}$, where we are assuming that $k\leq l$ to ensure
that the eigenvalues are simple. The convergence result proved at the
end of Step~3 then yields that, for $\ep$ small enough, $j$ large and
any fix compact set $K\subset \Om$, the difference 
\[
\|u^j_{k,\ep}- v_k\|_{C^m(K)}
\]
can be made as small as one wishes for $0\leq k\leq l$.

It is clear that the nodal set of the Neumann
eigenfunction $v_k$ is empty for $k=0$ and diffeomorphic to $k$ copies
of $\Si$ for $1\leq k\leq l$, namely,
\begin{equation}\label{copies}
v_k^{-1}(0)=\bigg\{ \frac{(1+2i-k)}{k}: i\in [0,k-1]\cap \ZZ\bigg\} \times\Si\,.
\end{equation}
Here we are identifying $\Om$ with $(-1,1)\times\Si$ as
before. Since $dv_k$ does not vanish on the nodal set and we have
shown that $u_{k,\ep}^j$ converges to $v_k$ in the $C^m$ norm on compact
subsets of $\Om$, Thom's isotopy theorem~\cite[Section 20.2]{AR}
implies that for small enough $\ep$, large $j$ and $k\leq l$, the nodal set $(u_{k,\ep}^j)^{-1}(0)$ has at least $k$
connected components, which are diffeomorphic to $\Si$. More
precisely, there is a diffeomorphism $\Phi^j_{k,\ep}$ of $M$, arbitrarily close
to the identity in any fixed $C^m(M)$ norm, such that $\Phi^j_{k,\ep}(v_k^{-1}(0))$ is a
collection of connected components of the nodal set
$(u_{k,\ep}^j)^{-1}(0)$. If $\Si$ separates, Courant's nodal domain
theorem then guarantees that the nodal set cannot have any further
components. 

For large enough $j$ and small $\ep$, let us consider the pulled-back
metric
\[
g:=(\Phi^j_{1,\ep})^*g_{j,\ep}
\]
and call $u_k$ its $k^{\mathrm{th}}$ eigenfunction, which is given by
\[
u_k=u^j_{k,\ep}\circ \Phi^j_{1,\ep}\,.
\]
The nodal set of $u_k$ has then a connected component given by $\Si$
if $k=1$ and a collection of components given by $\Phi_k(v_k^{-1}(0))$ if $2\leq k\leq l$, with
\begin{equation}\label{Phik}
\Phi_k:=(\Phi^j_{1,\ep})^{-1}\circ \Phi^j_{k,\ep}
\end{equation}
a diffeomorphism that can be taken as close to the identity in any fixed
$C^m$ norm as one wishes. We have therefore proved the following

\begin{theorem}\label{T.leigen}
If $\ep$ is small enough and $j$ is large enough, for $k\leq l$ the nodal sets of
the eigenfunctions $u_k$ corresponding to the metric $g$ contain a collection of connected components
diffeomorphic to the set~\eqref{copies}, which is given by $k$ copies
of $\Si$. The diffeomorphism, which is given by~\eqref{Phik}, can be taken arbitrarily close to the
identity in the $C^m(M)$ norm, and in fact is exactly the identity for
$k=1$. Furthermore, if $\Si$ separates, the
nodal set of $u_k$ does not have any additional components.
\end{theorem}

\section{Applications}
\label{S.applications}

The strategy of the proof of Theorem~\ref{T.leigen} is quite versatile
and can be employed to construct eigenfunctions with prescribed
behavior in different situations. We shall now present three concrete
applications of this philosophy, where we will consider eigenfunctions
with prescribed nodal sets of different topologies, low-energy
eigenfunctions with a large number of non-degenerate critical points
and nodal sets of Dirichlet eigenfunctions in manifolds with boundary.

\subsection{Eigenfunctions with nodal sets of different topologies}

In Theorem~\ref{T.leigen} we showed how to construct metrics whose
$k^{\mathrm{th}}$ eigenfunction has a nodal set diffeomorphic to $k$
copies of a given hypersurface $\Si$ (and possibly other connected
components, if $\Si$ does not separate). To complement this result, in
the following proposition we shall show how one can use the same
argument to construct a metric where the nodal set of the
$k^{\mathrm{th}}$ eigenfunction has a prescribed
connected component $\Si_k$, possibly not diffeomorphic for distinct
values of~$k$.

\begin{proposition}\label{P.different}
Let $\Si_1,\dots,\Si_l$ be pairwise disjoint closed orientable
hypersurfaces of $M$. Then for all $1\leq k\leq l$, there
exists a Riemannian metric $g$ on~$M$ such that $\Si_k$ is a
connected component of the nodal set $u_k^{-1}(0)$.
\end{proposition}

\begin{proof}
The proof is an easy modification of that of the main theorem. Indeed,
let us consider (pairwise disjoint) small neighborhoods $\Om_k$ of the orientable
hypersurfaces $\Si_k$. As before, we will identify $\Om_k$ with
$(-1,1)\times\Si_k$. The starting point now is a smooth reference
metric $g_0$ on $M$ taking the form
\[
g_0|_{\Om_k}=\Ga_k^2\, dx^2+g_{\Si_k}\,,
\]
in each domain $\Om_k$. Here $g_{\Si_k}$ is a metric on $\Si_k$ whose
first eigenvalue is assumed larger than $\pi^2$ and $\Ga_k$ are 
constants such that
\[
1=\Ga_1>\Ga_2>\cdots>\Ga_l>\frac12\,.
\]
Hence the first nonzero Neumann eigenvalue of $\Om_k$ is 
\[
\mu_k:=\frac{\pi^2}{4\Ga_k^2}\,,
\]
and the corresponding normalized eigenfunction is
\[
v_k:=\frac1{\sqrt{\Ga_k|\Si_k|}}\, \cos \frac{\pi(x+1)}2\,.
\]
Notice that here $v_k$ and
$\mu_k$ do {\em not}\/ have the same meaning as in
Section~\ref{S.main}, even though they will play totally analogous
roles in the proof.

We now set $\Om:=\bigcup_{k=1}^l \Om_k$ and consider the piecewise
smooth metric
\[
g_\ep:=\begin{cases}
g_0 & \text{in }\overline \Om\,,\\
\ep g_0 & \text{in }M\minus \Om\,.
\end{cases}
\]
Arguing exactly as in Steps~2 and~3 of the proof of the main theorem
with the new definitions of $\mu_k$ and $v_k$,
one concludes that the $k^{\mathrm{th}}$ eigenvalue $\la_{k,\ep}$ of
  the quadratic form defined by the metric $g_\ep$ tends to~$\mu_k$ as $\ep\searrow0$
  for all $1\leq k\leq l$, and that the corresponding eigenfunction
  $u_{k,\ep}$ converges to $v_k$ in $C^m(K)$, for $K$ any compact
  subset of $\Om_k$. 

  As in Step~4 of Section~\ref{S.main}, we can smooth the metric
  $g_\ep$ by introducing a sequence of smooth metrics
  $g_{j,\ep}$. Arguing as in Step~4, we infer that for large $j$ and
  small $\ep$ the nodal set of the eigenfunction $u^j_{k,\ep}$ then
  has a connected component diffeomorphic to $\Si_k$ for $1\leq k\leq
  l$, with the corresponding diffeomorphism $\Phi^j_{k,\ep}$ being
  arbitrarily close to the identity in any fixed $C^m$ norm. The point now is
  that, as the hypersurfaces $\Si_k$ are pairwise disjoint, the
  diffeomorphisms $\Phi^j_{k,\ep}$ can be assumed to differ from the
  identity only in a small neighborhood $S_k$ of each
  hypersurface. This allows us to define a diffeomorphism $\Phi^j_\ep$ of
  $M$ by setting $\Phi^j_\ep:=\Phi^j_{k,\ep}$ in each $S_k$ and letting
  $\Phi^j_\ep$ be the identity outside these sets.

The statement then follows by taking $g:=(\Phi^j_\ep)^*g_{j,\ep}$,
with large $j$ and small $\ep$, since in this case the eigenfunctions
are $u_k=u^j_{k,\ep}\circ \Phi^j_\ep$.
\end{proof}

\subsection{Critical points of low-energy eigenfunctions}

The study of level sets of a function is deeply connected with that of its critical points. Hence here we will use Morse theory and our construction
of eigenfunctions with prescribed nodal sets to show that there are
metrics whose low-energy eigenfunctions have an arbitrarily high
number of critical points. Notice that Theorem~\ref{T.cp} is a
corollary of Proposition~\ref{P.critical} below. Regarding the behavior of the critical
points of high-energy eigenfunctions, let us recall that
Jakobson and Nadirashvili~\cite{JN} proved that there are metrics
(even on the two-dimensional torus) for which the number of critical
points of the $k^{\mathrm{th}}$ eigenfunction does not tend to
infinity as $k\to\infty$.

\begin{proposition}\label{P.critical}
Given any positive integers $N$ and $l$, there is a smooth metric on~$M$ such that the $k^{\text{th}}$~eigenfunction $u_k$ has at least $N$
non-degenerate critical points for all $1\leq k\leq l$.
\end{proposition}

\begin{proof}
Let us fix some ball $B\subset M$ and take a domain $D$ whose closure
is contained in $B$. This ensures that $\Si:=\pd D$ separates. Theorem~\ref{T.leigen} shows that there is a
smooth metric $g$ such that the nodal set of its $k^{\mathrm{th}}$
eigenfunction $u_k$ is diffeomorphic to $k$ copies of
$\Si$. Furthermore, the corresponding eigenvalues $\la_k$ are simple
for $0\leq k\leq l$ and in Step~4 of Section~\ref{S.main} we showed that the
differential of $u_k$ does not vanish on its nodal set.

A theorem of Uhlenbeck~\cite{Uh76} ensures that one can take a
$C^{m+1}$-small perturbation $\tilde g$ of the metric $g$ so that the
first $l$ eigenfunctions are Morse, that is, all their critical points
are non-degenerate. Standard results from perturbation theory show
that the eigenfunctions $\tilde u_k$ of the perturbed metric are close
in the $C^m(M)$ norm to $u_k$, so in particular the nodal set of
$\tilde u_k$ has a connected component $\tilde\Si_k\subset B$ diffeomorphic to $\Si$ for all
$1\leq k\leq l$. Here we are using the fact that the gradient of $u_k$
does not vanish on its nodal set and Thom's isotopy theorem.

Call $D_k$ the domain contained in $B$
that is bounded by $\tilde\Si_k$ and let us denote by~$\tilde\nabla$ the covariant derivative associated
with the metric $\tilde g$. Since $\tilde\nabla \tilde u_k$ is a
nonzero normal vector on $\pd D_k$, which can be assumed to point
outwards without loss of generality, we can resort to Morse theory for
manifolds with boundary to show that the number of critical points of
$\tilde u_k$ of Morse index $i$ is at least as large as the
$i^{\text{th}}$ Betti number of the domain $\overline{D_k}$, for $0\leq i\leq
n-1$. Since $\overline{D_k}$ is diffeomorphic to $\overline{D}$, the proposition then follows
by choosing the domain $\overline{D}$ so that the sum of its Betti numbers is at
least $N$ (this can be done, e.g., by taking $\pd D$ diffeomorphic
to a connected sum of $N$ copies of any nontrivial product of spheres,
such as $\SS^1\times\SS^{n-2}$, since in this case the first Betti
number is~$N$).
\end{proof}

\subsection{Nodal sets of Dirichlet eigenfunctions}\label{SS.Payne}

Motivated by Payne's problem, in this subsection we will consider the nodal set of
Dirichlet eigenfunctions on manifolds with boundary. Hence here $M$ will be a compact manifold with boundary of
dimension $n\geq 3$. Notice that
Theorem~\ref{T.Payne} is a corollary of Proposition~\ref{P.Payne}
below. 

\begin{proposition}\label{P.Payne}
  Let $M$ be a compact $n$-manifold with boundary and let $\Si$ be a
  closed orientable hypersurface contained in the interior of
  $M$. Then there exists a Riemannian
  metric $g$ on~$M$ such that $\Si$ is a connected component of the nodal set of its second Dirichlet
  eigenfunction. If $\Si$
separates, then the nodal set is exactly~$\Si$.
\end{proposition}

\begin{proof}
The proof goes along the lines of that of the main
theorem. Specifically, the construction of the metrics $g_0$
and~$g_\ep$ is exactly as in Step~1 of Section~\ref{S.main}. The
variational formulation of the Dirichlet eigenvalues of the associated
quadratic form~$Q_\ep$ is analogous, the main difference being that
its domain is now $H^1_0(M)$. A minor notational difference is that,
since the first eigenvalue is customarily denoted in this context by $\la_1$ instead
of $\la_0$, $\cW_k$ (and $\cW_k'$ in Step~3) now
consist of $k$-dimensional, instead of $(k+1)$-dimensional,
subspaces. 

Step~2 also carries over verbatim to the case of manifolds with
boundary, with the proviso that the harmonic extension $\hat v_k$ must
also include a Dirichlet boundary condition at $\pd M$, i.e.,
\[
\De_0\hat v_k=0\quad\text{in } M\minus\BOm\,, \qquad \hat
v_k|_{\pd\Om}=v_k\,,\qquad \hat v_k|_{\pd M}=0\,.
\]
Steps~3 and~4 also remain valid in the case of manifolds with
boundary, the only difference being that the space $\cH_1$ must also
include Dirichlet boundary conditions:
\[
\cH_1:=\big\{\vp\in H^1_0(M): \De_0\vp=0\;\text{in }\Omc\big\}\,.
\]
The proposition is then proved using the same reasoning as in Section~\ref{S.main}, noticing that the diffeomorphism $\Phi^j_{1,\epsilon}$ can be assumed to differ from the
  identity only in a small neighborhood of the
  hypersurface $\Sigma$.
\end{proof}

\section*{Acknowledgments}

A.E.~and D.P.S.~are respectively supported by the Spanish MINECO through the
Ram\'on y Cajal program and by the ERC
grant~335079. This work is supported in part by the
grants~FIS2011-22566 (A.E.), MTM2010-21186-C02-01 (D.P.S.) and
SEV-2011-0087 (A.E.\ and D.P.S.).

\bibliographystyle{amsplain}

\begin{thebibliography}{99}\frenchspacing

\bibitem{AR}
R. Abraham and J. Robbin, {\it Transversal mappings and flows},
Benjamin, New York, 1967.

\bibitem{Alessandrini}
G. Alessandrini, Nodal lines of eigenfunctions of the fixed membrane problem in general convex domains, Comment. Math. Helv. 69 (1994) 142--154.

% \bibitem{Anne}
% C. Ann\'e, Spectre du laplacien et \'ecrasement d'anses,
% Ann. Sci. \'Ecole Norm. Sup. 20 (1987) 271--280.


\bibitem{BU83}
S. Bando, H. Urakawa, Generic properties of the eigenvalue of the Laplacian for compact Riemannian manifolds, Tohoku Math. J. 35 (1983) 155--172. 


\bibitem{Cheng}
S.Y. Cheng, Eigenfunctions and nodal sets, Comment. Math. Helv. 51
(1976) 43--55.

\bibitem{Colin}
Y. Colin de Verdi\`ere, Sur la multiplicit\'e de la premi\`ere valeur
propre non nulle du laplacien, Comment. Math. Helv. 61 (1986)
254--270.

\bibitem{DF}
H. Donnelly, C. Fefferman, Nodal sets of eigenfunctions on Riemannian manifolds, Invent. Math. 93 (1988) 161--183.

\bibitem{JDG}
A. Enciso, D. Peralta-Salas, Critical points of Green's functions on
complete manifolds, J. Differential Geom. 92 (2012) 1--29.

\bibitem{Adv}
A. Enciso, D. Peralta-Salas, Submanifolds that are level sets of solutions to a second-order elliptic PDE. Adv. Math. 249 (2013) 204--249.

\bibitem{EJN}
A. Eremenko, D. Jakobson, N. Nadirashvili, On nodal sets and nodal domains on $\SS^2$ and $\RR^2$,
Ann. Inst. Fourier (Grenoble) 57 (2007) 2345--2360. 

\bibitem{Freitas}
P. Freitas, Closed nodal lines and interior hot spots of the second
eigenfunction of the Laplacian on surfaces, Indiana Univ. Math. J. 51
(2002) 305--316.

\bibitem{HS}
R. Hardt, L. Simon, Nodal sets for solutions of elliptic equations,
J. Differential Geom. 30 (1989) 505--522.

\bibitem{Jak}
D. Jakobson, D. Mangoubi, Tubular neighborhoods of nodal sets and diophantine approximation,
Amer. J. Math. 131 (2009) 1109--1135.

\bibitem{JN}
D. Jakobson, N. Nadirashvili, Eigenfunctions with few critical points, J. Differential Geom. 53 (1999) 177--182.

\bibitem{Komen}
R. Komendarczyk, On the contact geometry of nodal sets, Trans. Amer. Math. Soc. 358 (2006) 2399--2413.

\bibitem{Lisi}
S.T. Lisi, Dividing sets as nodal sets of an eigenfunction of the Laplacian, Algebr. Geom. Topol. 11 (2011) 1435--1443.

\bibitem{Melas}
A.D. Melas, On the nodal line of the second eigenfunction of the Laplacian in $\RR^2$, J. Differential Geom. 35 (1992) 255--263.


\bibitem{Uh76}
  K. Uhlenbeck, Generic properties of eigenfunctions, Amer. J. Math. 98 (1976) 1059--1078.


\bibitem{Ya82}
S.T. Yau, Problem section, Seminar on Differential Geometry, Annals of Mathematics Studies 102 (1982) 669--706.

\bibitem{Ya93}
S.T. Yau, Open problems in geometry, Proc. Sympos. Pure Math. 54, pp. 1--28, Amer. Math. Soc., Providence, 1993. 




\end{thebibliography}

\end{document}